\documentclass[11pt]{article}
\usepackage{fullpage}

\usepackage{amsmath,amsfonts,amssymb,makeidx,graphicx,amsthm,natbib,caption,mathtools}

\usepackage{hyperref}
\hypersetup{colorlinks=true,
    citecolor=blue,
    urlcolor=blue
}

\usepackage{verbatim}
\newtheorem{theorem}{Theorem}
\newtheorem{lemma}[theorem]{Lemma}

\counterwithin{theorem}{section}

\def\F{\overline F}
\def\Povr{\buildrel P\over\rightarrow}
\def\CS{\cal S}

\usepackage{lipsum}

\newcommand\blfootnote[1]{%
  \begingroup
  \renewcommand\thefootnote{}\footnote{#1}%
  \addtocounter{footnote}{-1}%
  \endgroup
}

\counterwithin{equation}{section}

\def\ed{\end{document}}

\begin{document}
\begin{center}
\Large{\bf Large Deviation Probabilities for Sums of Random Variables with Heavy or Subexponential Tails}
\end{center}
\centerline{Daren B. H. Cline$^*$ and Tailen Hsing$^{**}$} 
\centerline{Texas A\&M University and University of Michigan}
\vskip .5cm
\begin{center}
\begin{minipage}{10cm}
\footnotesize
{
{\bf Abstract.} Let $S_n$ be the sum of independent random variables with
distribution $F$. Under the assumption that $-\log(1-F(x))$ is slowly
varying, conditions for 
$$
   \lim_{n\to\infty}\sup_{s\ge t_n}\left|{P[S_n>s]\over n(1-F(s))}-1\right|
   =0
$$ 
are given. These conditions extend and strengthen a series of previous
results.  Additionally, a connection with subexponential distributions
is demonstrated. That is, $F$ is subexponential if and only if the
condition above holds for some $t_n$ and 
$$
   \lim_{t\to\infty}{1-F(t+x)\over 1-F(t)} = 1
   \quad\text{for each real $x$.}$$
   }
\bigskip
\end{minipage}
\end{center}

%
\blfootnote{
 \par\hspace{.13cm}
 \par
$ ^*$  Research sponsored by National Science Foundation Grant 
DMS 9101083  \par
$^{**}$Research sponsored by National Science Foundation Grant
DMS 9107507 
}

\vspace{-1cm}
\section{Subexponentiality and large deviations}
Let $X_i,\ i\ge 1$, be independent and identically distributed
random variables with distribution function $F$, and
$S_n=\sum_{i=1}^nX_i$.  Suppose $\{t_n\}$ is a sequence of
positive constants such that ${S_n/ t_n}\Povr 0$.  
It is often of interest to consider the rate
at which the large deviation probabilities $P[|S_n|>t_n]$ and
$P[S_n>t_n]$ tend to zero.  This topic is of traditional
importance in probability theory, and has numerous statistical applications.

Following Cram\'er's pioneering work (1938), the study of large
deviation probabilities at first was confined to
distributions $F$ satisfying Cram\'er's condition, namely,
for some constant $\epsilon > 0$,
\begin{align} \label{e:1.1}
\int_{-\infty}^\infty e^{cx}F(dx)<\infty\quad\text{for all
$c\in[-\epsilon,\epsilon]$}.
\end{align}
For example, \eqref{e:1.1} implies
$$
\lim_{n\to\infty} n^{-1} \log P[S_n > nx] = \sup_{\lambda\ge 0}
[\lambda x - L(\lambda)] \quad \text{for all } x > E(X_1),
$$
and
$$
\lim_{n\to\infty} n^{-1} \log P[S_n < nx] = \sup_{\lambda \le 0}
[\lambda x - L(\lambda)] \quad \text{for all } x < E(X_1),
$$
where $L(\lambda) = \log Ee^{\lambda X_1}$.

It was not until the 1960's that attention was given to
heavy--tailed distributions.  In that connection, the
most noticeable work was by Linnik (1961), Heyde (1967a,b,1968), A.V.
Nagaev (1969a,b) and S.V. Nagaev (1973).  (See also the review paper by S.V.
Nagaev (1979).) While Linnik, A. V. Nagaev, and S. V. Nagaev mostly considered
distributions with finite variance, Heyde focused on
distributions with infinite variance, including the ones
that are attracted to nonnormal stable laws.
However, their results contain a common implication for the heavy
tailed distributions they considered: namely, if $t_n$ tends to $\infty$ fast
enough then one has
\begin{align} \label{e:1.2}
0<\liminf_{n\to\infty}\dfrac{P[S_n>t_n]}{nP[X>t_n]}
\le \limsup_{n\to\infty}\dfrac{P[S_n>t_n]}{nP[X>t_n]}<\infty
\end{align}
and, in some cases, even
\begin{align} \label{e:1.3}
\lim_{n\to\infty}\dfrac{P[S_n>t_n]}{nP[X>t_n]} = 1.
\end{align}
Typically the results require at least $n(1-F(t_n))\to 0$.
Note that if this is so then \eqref{e:1.2} and \eqref{e:1.3} state that the
probabilities of large deviations for the sum and the maximum are 
asymptotically comparable.

What class of distributions can be expected to have these properties?
The most natural notion is that of subexponentiality. Specifically,
the subexponential class $\CS$ consists of those
probability distributions $F$ satisfying
\begin{align} \label{e:1.4}
&\lim_{t\to\infty} {P[X_1 + X_2 > t] \over \F(t)} \quad \text{exists
finite}
\end{align}
and 
\begin{align} \label{e:1.5}
\lim_{t\to \infty} {\F(t+x) \over \F(t)} = 1 \quad \text{for each real} \ x, 
\end{align}
where $\F(x)=1-F(x)$ and $X_1$ and $X_2$ are independent random
variables distributed according to $F$.  This class was first studied
by Chistyakov (1964) and by Chover, Ney and Wainger (1973a) in the
case $P[X \geq 0] = 1$, with application to branching processes.
Additional work is found in Embrechts and Goldie (1980) and Cline
(1987).  Extensions for the case $P[X\ge 0]<1$ are given by Willekens
(1986) and numerous references and applications are provided in
Embrechts and Goldie (1982).

Two properties of $\CS$ are of interest to us here.
First, if $F \in \CS$ then \eqref{e:1.1} fails to hold (cf.\ Embrechts and Goldie
(1982)). Second, the limit in \eqref{e:1.4} equals $2$ and furthermore
\begin{align} \label{e:1.6}
\lim_{t\to \infty} {P[S_n > t] \over n\F(t)} = 1 \quad \text{for every}\
n\ge 1
\end{align}
(cf.\ Chistyakov (1964), with extension by Willekens (1986)).
Immediately one has that, for some sequence $\{t_n\}$, \eqref{e:1.3} holds for
distributions in the subexponential family.  In
fact, we may characterize $\CS$ as follows.

\begin{theorem} \label{thm:1.1} $F \in \CS$ if and only if \eqref{e:1.5} holds and there exists
$t_n$ such that
\begin{align} \label{e:1.7}
\limsup_{n \to \infty} \sup_{s \geq t_n} {P[S_n > s] \over n\F(s)} \leq 1.
\end{align}
\par \noindent
Moreover, when $F\in\CS$ there exists $t_n$ such that
\begin{align} \label{e:1.8}
\lim_{n \to \infty} \sup_{s \geq t_n} \left| {P[S_n > s] \over n\F(s)}-1 \right|
= 0.
\end{align}
\end{theorem}
\vskip4mm
Our proof of Theorem \ref{thm:1.1} (in Section 4), however, does not construct
a specific choice of $\{t_n\}$ for (1.8) to hold.  Therefore our second 
goal is to consider how to choose the constants $t_n$ for a particular 
subexponential distribution. 

The subexponential distributions that we will focus on are those $F$
for which $-\log\F(x)$ is slowly varying as $x\to\infty$.
These are the distributions whose tails are on the heavy side in $\CS$, and
they include distributions that are attracted to the nonnormal stable
distributions as well as distributions with finite variance, such as the
lognormal. Large deviation results for this class are few and scattered.
For example, there are no results for distributions whose tails behave like
the Cauchy and lognormal distributions.
We show that for these distributions, it is possible to have a unified 
treatment based on a truncation argument and Markov's inequality. We not only 
obtain non-trivial extensions of results previously proved by using different 
methods, but we also fill some holes in the literature. The main result
is Theorem \ref{thm:2.1} in the next section which is followed by a
discussion on the conditions of the result. In Section 3 we give some
examples that can be derived directly from Theorem \ref{thm:2.1}, or from slight 
modifications of Theorem \ref{thm:2.1}. We also describe how existing results relate 
to the examples.

The paper does not cover ``semiexponential'' distributions, e.g.
$F(x)=1-e^{-x^\rho}$ with $0 < \rho < 1$, which are also in the
subexponential class (cf.\ Cline (1986) and Goldie and Resnick
(1988)). A.V. Nagaev (1969b) and S.V. Nagaev (1973) have considered
large deviations for special cases of semiexponential distributions.
We do not know of a unified approach that will work for all
subexponential distributions. 

For clarity, all the proofs and technical details are collected in Section 4.
\vskip5mm
\section{Main Result} For what follows, let $\psi(t)=-\log\F(t)$.
When \eqref{e:1.5} holds, $\F(\log t)$ is slowly varying and thus $\psi$ has
the representation
\begin{align} \label{e:2.1}
   \psi(t)=b(t)+\int_0^t\eta(u)du,\qquad{\rm for\ all\ } t\ge0,
\end{align}
where $\eta(t)$ and $b(t)$ are measurable and $\eta(t)\to 0$, 
$b(t)\to b\in\mathbb{R}$, as $t\to\infty$. Because of this, a represenation
like \eqref{e:2.1} will be the starting point for our theorems.

Also, define
$$\mu_1(t) = \int_{|x|\le t}xdF(x)\quad\text{and}\quad
  \mu_2(t) = \int_{|x|\le t}x^2dF(x),\quad t>0.$$

We now state the main result of this paper. 
\vskip4mm\noindent
\begin{theorem}\label{thm:2.1} Let $a=a_n(s)=\psi(s)-\log n$. Assume (2.1)
where $\eta(t)\downarrow 0$ and $b(t)$ is measurable, bounded and satisfies
\begin{align} \label{e:2.2}
   \lim_{y\to 1,t\to\infty}(b(yt)-b(t))=0.
\end{align}
Suppose $t_n$ are constants increasing to $\infty$ and
$\lambda\in(0,1)$ such that, as $n\to\infty$,
\vskip4pt
\begin{align} \label{e:2.3}
\sup_{s\ge t_n}\inf_{w\ge s/a}n\left({|\mu_1(w)|+a\mu_2(w)/s\over
s\wedge\eta^{-1}(\lambda s)}+F(-w)\right)\to 0,  
\end{align}
\vskip4pt
\begin{align} \label{e:2.4}
   \sup_{s\ge t_n} 1\vee s\eta(\lambda s)\log s\eta(\lambda s)
\end{align}
and
\begin{align} \label{e:2.5}
   \sup_{s\ge t_n}n(1\vee s\eta(\lambda s))\F(s/a)\to 0,
\end{align}
where $\eta^{-1}(\lambda s) = 1/\eta(\lambda s)$.
Then
\begin{align} \label{e:2.6}
   \lim_{n\to\infty}\sup_{s\ge t_n}\left|{P[S_n>s]\over n\F(s)}-1\right|=0.
\end{align}
\end{theorem}

\vskip4mm\noindent
The conditions \eqref{e:2.3}-\eqref{e:2.5} attempt to cover a variety of situations.
While they do not offer much insight into the large deviation problem at
this stage, it is worth mentioning that they are sharp since they form the
weakest condition for \eqref{e:2.6} to hold in the case where the tail
probabilities of $F$ are regularly varying with index in $(-1,0]$
(cf.\ Theorem \ref{thm:3.3}). The following remarks are useful.

\setcounter{equation}{2} 
\renewcommand{\theequation}{2.\arabic{equation}'}

\vskip4mm\noindent
{\bf Remark 1.} Note that \eqref{e:2.3} is equivalent to
\begin{align} \label{e:2.3'}
\sup_{s\ge t_n}n\left({|\mu_1(w)|+a\mu_2(w)/s\over
s\wedge\eta^{-1}(\lambda s)}+F(-w)\right)\to 0
\text{ for some } w = w_n(s) \ge s/a. 
\end{align}
\setcounter{equation}{6}
\counterwithin{equation}{section}
\vskip4mm\noindent
{\bf Remark 2.} The conditions require
$$
nP[|X|>w_n]\to 0, \ {n\mu_1(w_n)\over t_n\wedge\eta^{-1}(\lambda t_n)}\to 0,
\text{ and } {n\mu_2(w_n)\over (t_n\wedge\eta^{-1}(\lambda t_n))^2}\to 0,
$$
for some $w_n$ and hence
$${S_n\over t_n\wedge\eta^{-1}(\lambda t_n)}\Povr 0,$$
by a standard argument using truncation and Chebyshev's inequality.
\vskip4mm\noindent
{\bf Remark 3.} Condition \eqref{e:2.4} is possible for some $t_n$ if and only if
$(t\eta(t)\log t\eta(t))/\psi(t)\to 0$, as $t\to\infty$, which
implies that $\psi$ is slowly varying. Condition \eqref{e:2.5} is possible 
for some $t_n$ if and only if $t\eta(t)\F(t/\psi(t))\to 0$. The 
following lemma gives sufficient conditions.
\vskip4mm\noindent
\begin{lemma} \label{lm:2.2}
Assume \eqref{e:2.1} with $\eta(t)\downarrow 0$ and $b(t)\to b$, as
$t\to\infty$. 
\item{(i)} If $\exp(\log^2\psi(t))$ is slowly varying then
$$\lim_{t\to\infty}{t\eta(t)\log\psi(t)\over\psi(t)}=0.$$
\item{(ii)} If
\begin{align} 
   \lim_{t\to\infty}{t\eta(t)\over\psi(t)} =0 \label{e:2.7} 
\end{align}
and
\begin{align} \label{e:2.8}   
\liminf_{t\to\infty}{\eta(t/\psi(t))\over\eta(t)}>1 
\end{align}
then
$$
   \lim_{t\to\infty}t\eta(t)\F(t/\psi(t))=0.
$$
\end{lemma}

\vskip4mm\noindent
{\bf Remark 4.} In particular, \eqref{e:2.7} and \eqref{e:2.8} hold if $\eta$ is monotone and
regularly varying with index $-1$. Some examples are given in Section 3.
\vskip5mm
\section{Examples and ramifications}
For our first example we consider a class of finite variance distributions 
for which the condition on $t_n$ is simply expressed. This class 
includes the lognormal distribution and some log-Weibull distributions.

\vskip4mm\noindent
\begin{theorem} \label{thm:3.1} Suppose $F$ has mean 0 and a finite $2+\delta$ moment
for  some positive $\delta$. Suppose also that \eqref{e:2.1} and \eqref{e:2.2} hold with 
$b$ bounded and measurable, $\eta(t)\downarrow 0$, $t\eta(t)$ bounded away 
from 0, and
\begin{align} \label{e:3.1}
\lim_{t\to\infty}{t\eta(t)\log t\eta(t)\over\psi(t)}=0.
\end{align}
If $\lambda\in(0,1)$ and $t_n\uparrow\infty$ satisfy
\begin{align} \label{e:3.2}
\sup_{s\ge t_n}{n\psi(s)\eta(\lambda s)\over s}\to 0,
\end{align}
as $n\to\infty$, then \eqref{e:2.6} holds.
\end{theorem}
\vskip4mm\noindent
For example, suppose $F$ is the centered lognormal distribution
$$
F(x) = \Phi\left({\log(x+\beta)\over\sigma}\right),\qquad x>-\beta,
$$
where $\Phi$ is the standard normal distribution and $\beta=e^{\sigma^2/2}$.
Then
$$
\F(x)
 \sim{\sigma\over\sqrt{2\pi}}e^{-(\log(x+\beta))^2/2\sigma^2-\log\log(x+\beta)}
 \sim{\sigma\over\sqrt{2\pi}}e^{-(\log x)^2/2\sigma^2-\log\log x}.
$$
and we can take
$$
\eta(x)={\log x\over\sigma^2x}+{1\over x\log x}\sim{\log x\over\sigma^2x}.
$$
Hence if $t_n$ is such that
$$
 \lim_{n\to\infty}{n\log^3t_n\over t_n^2}=0
$$
then \eqref{e:2.6} holds.
\vskip4mm
If $b$ is bounded in Theorem 2.1 but does not satisfy \eqref{e:2.2} then $F$ may not
be subexponential. Nevertheless, useful limiting bounds are available. From 
the proof of Theorem 2.1 one may deduce that if all the other assumptions
are met then
\begin{align}\label{e:3.3}
e^{-\gamma}\le\liminf_{n\to\infty}\inf_{s\ge t_n}{P[S_n>s]\over n\F(s)}
  \le\limsup_{n\to\infty}\sup_{s\ge t_n}{P[S_n>s]\over n\F(s)}\le e^\gamma,
\end{align}
where 
$$
   \gamma=\lim_{y\downarrow 1}\limsup_{t\to\infty}
      \sup_{1\le x\le y}(b(xt)-b(t)).
$$

In a slightly different direction, the proof of Theorem \ref{thm:2.1} can be
specialized to handle distributions with dominated varying tails. To
that end, we recall some basic definitions (cf.\ Bingham, Goldie, and
Teugels (1987), p. 61--76, and Cline (1994)).

For a distribution $F$, $\F$ is {\it regularly varying} 
$(\F\in RV_{-\alpha})$ if, for some nonnegative $\alpha$,
$$
   \lim_{t\to\infty}{\F(\lambda t)\over\F(t)}=\lambda^{-\alpha},
   \quad\text{for all $\lambda\ge 1$}.
$$
$\F$ is {\it intermediate regularly varying} if
$$
   \lim_{\lambda\to 1,t\to\infty}{\F(\lambda t)\over\F(t)}=1.
$$
$\F$ is {\it dominated varying} if
$$
   \liminf_{t\to\infty}{\F(\lambda t)\over\F(t)}>0,
     \quad\text{for some $\lambda>1$}.
$$

\vskip4mm\noindent
\begin{theorem} \label{thm:3.2} Suppose $P[|X_1|>x]$ is dominated varying and 
$S_n/t_n\Povr 0$. If
\begin{align} \label{e:3.4}
\limsup_{n\to\infty}\sup_{s\ge t_n}{na\mu_2(s)\over s^2}=0
\end{align}
then \eqref{e:3.3} holds with
$$
   \gamma=-\log\lim_{\lambda\downarrow 1}\liminf_{t\to\infty}
      {\F(\lambda t)\over\F(t)}.
$$
Moreover, if $F$ is intermediate regular varying then \eqref{e:2.6} holds.
\end{theorem} 

\vskip4mm\noindent
Theorem \ref{thm:3.2} shows that distributions with
intermediate regularly varying tails are subexponential. In fact it is
known that the subexponential class contains those $F$ with dominated
varying right tails and satisfying \eqref{e:1.5} (Goldie (1978)).
\vskip4mm\noindent
\begin{theorem} \label{thm:3.3}  Suppose $P[|X|>x] \in RV_{-\alpha}$,
where $\alpha\ge 0$, $p=\sup_{x\ge 0}\dfrac{F(-x)}{\F(x)}<\infty$
and $t_n$ satisfies $n\F(t_n)\rightarrow 0$.  Then any one of the
following implies \eqref{e:2.6}:
\begin{itemize}
\item[(i)] $0\le\alpha<1$.
\vskip4pt
\item[(ii)] $1\le\alpha<2$ and 
{$\displaystyle \lim_{n\to\infty}{n\over t_n}\mu_1(t_n)=0$.}
\vskip4pt
\item[(iii)] $\alpha\ge 2$ and 
{$\displaystyle \lim_{n\to\infty}{nE(X)\over t_n}=
\lim_{n\to\infty}{n\mu_2(t_n)\log t_n\over t_n^2}=0$.}
\end{itemize}
\end{theorem}

\vskip4mm\noindent
In connection with Theorem 3.3, Heyde (1968) considered the infinite variance
case and Nagaev (1969b) considered the finite variance case. Our method, which
is a mixture of theirs, has several advantages. First we are able to
study these two cases in a unifying way, whereas the methods in Heyde (1968)
and Nagaev (1969b) are essentially different. For the infinite variance case,
our result is new since Heyde (1968) only considered large deviations of 
$|S_n|$. Also neither author considered the cases where $\alpha = 0, 1$ or 2.
\vskip5mm

\section{Proofs}
\vskip4mm\noindent
{\bf proof of Theorem \ref{thm:1.1}.}  If $F \in \CS$, then \eqref{e:1.5} holds by definition.
To show both \eqref{e:1.7} and \eqref{e:1.8} we simply choose $t_n$ by \eqref{e:1.6} to satisfy
$$
\sup_{s\ge t_n} \left| {P[S_n > s] \over n\F(s)} -1\right| \leq {1 \over n}.
$$
\par
Suppose instead that \eqref{e:1.5} holds and such $t_n$ exist that \eqref{e:1.7} holds.
Choose $n_0$ so that for each $n \geq n_0$,
$$
\sup_{s \geq t_n} {P[S_n > s] \over n\F(s)} \leq 1 + \epsilon.
$$
\par \noindent
Let $F_n(x) = P[S_n \leq x]$. Equation \eqref{e:1.5} is also satisfied with $F_n$
replacing $F$ (cf.\ \hbox{Embrechts} and Goldie (1980); with extension by
Willekens (1986)).
Using this fact and Fatou's Lemma,
\begin{align} \label{e:4.1}
\begin{split}
\liminf_{x \to \infty} {\F_{2n}(x) \over \F_2(x)}
&\ge\liminf_{x\to\infty}\int_{-\infty}^{x/2}{\F_{2n-2}(x-u)\over\F_2(x)}F_2(du)
+\liminf_{x\to\infty}\int_{-\infty}^{x/2}{\F_2(x-u)\over\F_2(x)}F_{2n-2}(du) \\
&\geq\liminf_{x \to \infty} {\F_{2n-2}(x) \over \F_2(x)} + 1.
\end{split}
\end{align}
By induction, $\liminf_{x \to \infty} {\F_{2n}(x) \over n \F_2(x)} \geq 1.$
Choose successively, therefore, $y_n \ge \max(t_{2n},y_{n-1})$ so that
$$
\inf_{y \geq y_n} {\F_{2n}(y) \over n\F_2(y)} \geq 1 - \epsilon.
$$
\par
Now let $x_k \to \infty$ and define $m = \sup \{n: x_k \geq y_{n} \}$. Then
for $k$ large enough so that $2m > n_0$,
$$
{\F_2(x_k) \over 2\F(x_k)}=
{m\F_2(x_k)\over\F_{2m}(x_k)}{\F_{2m}(x_k)\over 2m\F(x_k)}
\leq{1 + \epsilon \over 1 - \epsilon}.
$$
That is,
$$
\limsup_{k \to \infty} {\F_2(x_k) \over 2\F(x_k)} \leq 1.
$$
\par \noindent
By the same use of Fatou's Lemma that gave \eqref{e:4.1},
$$
\liminf_{k \to \infty} {\F_2(x_k) \over 2\F(x_k)} \geq 1.
$$
\par \noindent
Since the sequence $x_k$ is arbitrary, then \eqref{e:1.4} holds and $F\in\CS$.
\qed
\vskip4mm

We need the following lemma in the proof of Theorem \ref{thm:2.1}.
\vskip4mm\noindent
\begin{lemma} \label{lm:4.1} 
Let $1\le y_n\le x_n$. There exists $z_n$ such that
$x_n/z_n\to\infty$, $z_n/y_n\to\infty$ and $x_n/(y_n\log z_n)\to\infty$ 
if and only if $x_n/y_n\to\infty$ and $x_n/(y_n\log y_n)\to\infty$.
\end{lemma}

\begin{proof} The sufficiency is obvious. For the necessary part, let
$$
   v_n=\min\left(\left({x_n\log y_n\over y_n}\right)^{1/2}\!\!\!,\,\,\,
   {1\over 2}(\log x_n+\log y_n)\right)
$$
and $z_n=e^{v_n}$.
\end{proof}
\vskip4mm\noindent
{\bf Proof of Theorem \ref{thm:2.1}.}
We shall proceed by first proving
\begin{align} \label{e:4.2}
   \limsup_{n\to\infty}\sup_{s\ge t_n}{P[S_n>s]\over n\F(s)}\le 1
\end{align}
and then
\begin{align} \label{e:4.3}
   \liminf_{n\to\infty}\inf_{s\ge t_n}{P[S_n>s]\over n\F(s)}\ge 1.
\end{align}

Let $\epsilon=\epsilon_n(s)$ and $\epsilon'=\epsilon'_n(s)$ be
functions which vanish uniformly in $s\ge t_n$, as $n\to\infty$. Specific
choices for $\epsilon$ and $\epsilon'$ will be made later. Define
$$
   m=m_n(s)=(1-\epsilon')s.
$$

Let $S_n'=\sum_{i=1}^n X_i1_{X_i\le m}$. Then
\begin{align} \label{e:4.4}
\begin{split}
{P[S_n>s]\over n\F(s)}
   &\le {nP[X_1>m]+P[S_n'>s]\over n\F(s)} \\
   &= {\F(m)\over \F(s)}+{P[S_n'>s]\over n\F(s)}.
\end{split}
\end{align}
By Markov's inequality,
$$
   {P[S_n'>s]\over n\F(s)}
   \le\exp\left(n\int_{-\infty}^m(e^{cx}-1)F(dx)-cs+a\right)
$$
for any $c>0$. In particular, the choice $c=(1+\epsilon)a/s$ gives
\begin{align} \label{e:4.5}
\begin{split}
  {P[S_n'>s]\over n\F(s)}
   &\le\exp\left(n\int_{-\infty}^m(e^{(1+\epsilon)ax\over s}-1)F(dx)
      -\epsilon a\right) \\
   &\le\exp\left(n\int_{-w}^{s/a}(e^{(1+\epsilon)ax\over s}-1)F(dx)
      +n\int_{s/a}^m(e^{(1+\epsilon)ax\over s}-1)F(dx)
      -\epsilon a\right)
\end{split}
\end{align}
for any $w > 0$. In the following, let $w = w_n(s)$ be chosen by \eqref{e:2.3'}.
In view of \eqref{e:4.4} and \eqref{e:4.5}, condition \eqref{e:4.2} follows from choosing $\epsilon$ and
$\epsilon'$ so that
\begin{align} \label{e:4.6}
   \lim_{n\to\infty}\sup_{s\ge t_n}{\F(m)\over\F(s)}= 1,
\end{align}
\begin{align} \label{e:4.7}
   \limsup_{n\to\infty}\sup_{s\ge t_n}{n\over\epsilon a}
      \int_{-w}^{s/a}(e^{(1+\epsilon)ax\over s}-1)
      F(dx) \le \delta 
\end{align}
\begin{align} \label{e:4.8}
   \limsup_{n\to\infty}\sup_{s\ge t_n}{n\over \epsilon a}\int_{s/a}^m
(e^{(1+\epsilon)ax\over s}-1) F(dx) < 1- \delta
\end{align}
and
\begin{align} \label{e:4.9}
\epsilon a \to \infty \text{ as } n\to\infty 
\end{align}
for some $\delta\in (0,1)$.  To this end, define
$$
   B=1+\sup_{u\ge t\ge 0}|b(t)-b(u)|
$$
and choose $z_n(s)$ according to Lemma \ref{lm:4.1} with
$x_n(s)=a_n(s)$ and $y_n(s)=1\vee s\eta(\lambda s)$. Fix $\delta\in (0,
1-e^{-2})$ and define
\begin{align} \label{e:4.10}
   \epsilon=\max\left(z_n(s)^{-1},ne^{B+3}\F(s/a),
	 {2n\over\delta}\left({|\mu_1(w)|\over s}
	 + {e^{2}\mu_2(w)\over s^2/a}\right)
	 \right).
\end{align}
Using \eqref{e:2.3'}, \eqref{e:2.4} and \eqref{e:2.5} and Lemma \ref{lm:4.1} we have, uniformly for $s\ge t_n$,
\begin{align} \label{e:4.11}
\begin{split}
   (1\vee s\eta(\lambda s))\epsilon
   &\le\max\Bigg({y_n(s)\over z_n(s)},
   ne^{B+3}(1\vee s\eta(\lambda s))\F(s/a),
  {2n(|\mu_1(w)|+e^2a\mu_2(w)/s)\over
	   \delta(s\wedge\eta^{-1}(\lambda s))}\Bigg)\\
   &\to 0, 
   \end{split}
   \end{align}
\begin{align} \label{e:4.12}
\epsilon a\ge{x_n(s)\over z_n(s)}&\to\infty
\end{align}
and
\begin{align} \label{e:4.13}
{(-\log\epsilon)(1\vee s\eta(\lambda s))\over a}
   &\le{y_n(s)\log z_n(s)\over x_n(s)}\to 0.
\end{align}
Now choose
$$
   \epsilon'={\epsilon+(2B-\log\epsilon)/a\over 1-s\eta(\lambda s)/a}.
$$
Thus
\begin{align} \label{e:4.14}
   (\epsilon'-\epsilon)a-\epsilon's\eta(\lambda s)-2B=-\log\epsilon.
\end{align}
In addition, from \eqref{e:2.4}, \eqref{e:4.11} and \eqref{e:4.13} we have $\epsilon'\to 0$
and with the monotonicity of $\eta$,
\begin{align} \label{e:4.15}
s\eta(m)\epsilon'&\le(1\vee s\eta(\lambda s))\epsilon'\cr
   &={(1\vee s\eta(\lambda s))(\epsilon+(2B-\log\epsilon)/a)
   \over 1-s\eta(\lambda s)/a}\to 0.
\end{align}

We now show \eqref{e:4.6}--\eqref{e:4.8}, since \eqref{e:4.9} already follows from \eqref{e:4.12}. From \eqref{e:4.15}
and the assumption on $b$,
\begin{align} \label{e:4.16}
   \psi(s)-\psi(m)\le b(s)-b(m)+\epsilon's\eta(m)\to 0, 
\end{align}
uniformly for $s\ge t_n$, which is \eqref{e:4.6}. Using $w\ge s/a$, \eqref{e:4.10}
and a Taylor expansion, for large enough $n$,
$$
   n\int_{-w}^{s/a}(e^{(1+\epsilon)ax\over s}-1)
     F(dx)
   \le 2n{\mu_1(w)\over s/a}
   +2ne^2{\mu_2(w)\over (s/a)^2}\le \delta a\epsilon,
$$
uniformly for $s\ge t_n$, which is \eqref{e:4.7}.

Integrating by parts, we get
\begin{align} \label{e:4.17}
n\int_{s/a}^m (e^{(1+\epsilon)ax\over s}-1)F(dx)
   \le n{(1+\epsilon)a\over s}
      \int_{s/a}^me^{{(1+\epsilon)ax\over s}-\psi(x)}\,dx
      +ne^{(1+\epsilon)}\F(s/a). 
\end{align}
Since $\eta$ is decreasing, ${(1+\epsilon)ax\over s}-(\psi(x)-b(x))$ is convex
and hence
\begin{align} \label{e:4.18}
\sup_{s/a\le x\le m} e^{{(1+\epsilon)ax\over s}-\psi(x)}
\le \max\left(e^{{(1+\epsilon)ma\over s}-\psi(m)+B-1},
    e^{B+\epsilon}\F(s/a)\right).
\end{align}
By \eqref{e:4.17}, \eqref{e:4.18} and the fact that $(1+\epsilon)m\le s$, we obtain the bound
\begin{align} \label{e:4.19}
\begin{split}
n\int_{s/a}^m (e^{(1+\epsilon)ax\over s}-1)F(dx)
   &\le na\max\left(e^{{(1+\epsilon)ma\over s}-\psi(m)+B-1},
      e^{B+\epsilon}\F(s/a)\right)+ne^{(1+\epsilon)}\F(s/a)\\
   &\le a\max\left(e^{-(\epsilon'-\epsilon)a+\epsilon's\eta(\lambda s)+2B-2},
      e^{B+1}n\F(s/a)\right)
      +ne^2\F(s/a)\\
   &=e^{-2}\epsilon a+ne^2\F(s/a),
   \end{split}
\end{align}
where the final equality comes from \eqref{e:4.10} and \eqref{e:4.14}. The bound in \eqref{e:4.8} follows
from \eqref{e:2.5} and that completes the proof of \eqref{e:4.2}.

Finally, we verify \eqref{e:4.3}. Let $m'=s+\zeta(s\wedge\eta^{-1}(\lambda s))$
where $\zeta$ is any
fixed positive constant. By Bonferroni's inequality,
\begin{align} \label{e:4.20}
\begin{split}
 P[S_n>s]&\ge P[S_n>s,\ \max_{1\le i\le n}X_i>m'] \\
   &\ge\sum_{i=1}^n P[S_n>s,\ X_i>m']
      \,-\!\!\!\sum_{1\le i<j\le n}\!\!P[S_n>s,\ X_i>m',\ X_j>m']\\
   &\ge n\F(m')\left(P[S_{n-1}>-\zeta(s\wedge\eta^{-1}(\lambda s))]
   -{n\over 2}\F(s)\right).
   \end{split}
\end{align}
Note that $n\F(s)\to 0$ by \eqref{e:2.4}, and if we let
$$\gamma(\zeta)=\limsup_{t\to\infty}\sup_{0\le u\le\zeta}(b((1+u)t)-b(t))$$
then
\begin{align} \label{e:4.21}
  \liminf_{n\to\infty}\inf_{s\ge t_n}{\F(m')\over \F(s)}
    \ge e^{-\zeta-\gamma(\zeta)}
\end{align}
(cf.\ the derivation of \eqref{e:4.16}).

By Remark 2 of Section 2,
$$
\sup_{s\ge t_n} {S_{n-1}\over s\wedge \eta^{-1}(\lambda s)} \Povr 0.
$$
Hence
$$
\inf_{s\ge t_n}P[S_{n-1}\ge -\zeta(s\wedge\eta^{-1}(\lambda s))] \to 1
$$
and by \eqref{e:4.20} and \eqref{e:4.21},
\begin{align*}
\liminf_{n\to\infty}\inf_{s\ge t_n}{P[S_n>s]\over n\F(s)}
  &\ge\liminf_{n\to\infty}\inf_{s\ge t_n}{\F(m')\over \F(s)} 
   \left(P[S_{n-1}>-\zeta(s\wedge\eta^{-1}(\lambda s))]
   -{n\over 2}\F(s)\right) \\
  &\ge e^{-\zeta-\gamma(\zeta)}.
\end{align*}
Condition \eqref{e:4.3} follows from this since $\zeta > 0$ is arbitrary and
$\gamma(\zeta)\to 0$ as $\zeta\to 0$ by \eqref{e:2.2}.\hfil\break
\qed
\vskip4mm\noindent
{\bf Proof of Lemma \ref{lm:2.2}.}
Define $B_1=\inf_{u\ge 0}b(u)$ and $\psi_1(t)=\psi(t)-b(t)+B_1$.
Also define $B_2=\sup_{u\ge 0}b(u)$ and $\psi_2(t)=\psi(t)-b(t)+B_2$. 

(i) Let $\xi(t)=\exp(\log^2\psi_1(t))$. Then, for large enough $t$,
\begin{align*}
   {2\eta(2t)\log\psi(2t)\over\psi(2t)}
   &\le{2\over\xi(t)}\int_t^{2t}
   {\eta(u)(\log\psi_1(u))\xi(u)\over\psi_1(u)}\,du\cr
   &={\xi(2t)\over\xi(t)}-1\\
   &\le{e^{\log^2\psi(2t)}\over e^{\log^2\psi(t)}}
   e^{-2\log\psi(t)\log(1-(B_2-B_1)/\psi(t))}-1\to 0.
\end{align*}

(ii) From \eqref{e:2.7} and \eqref{e:2.8} choose $t_0$ so that
$$
   \left(1-{t\eta(t)\over\psi_2(t)}\right){\eta(t/\psi_2(t))\over\eta(t)}\ge 1,
   \qquad\qquad\text{for\ all\ $t\ge t_0$.} 
$$
Then
\begin{align*}
   \psi(t/\psi(t))&\ge\psi_2(t/\psi_2(t))+B_1-B_2\\
   &=\psi_2(t_0/\psi_2(t_0))+B_1-B_2
    +\int_{t_0}^t\left(1-{u\eta(u)\over\psi_2(u)}
   \right){\eta(u/\psi_2(u))\over\psi_2(u)}\,du\\
   &\ge\psi_2(t_0/\psi_2(t_0))+B_1-B_2
     +\int_{t_0}^t{\eta(u)\over\psi_2(u)}\,du\\
   &\ge\psi_2(t_0/\psi_2(t_0))+B_1-B_2-\log\psi_2(t_0)+\log\psi(t).
\end{align*}
Hence
$$
   t\eta(t)\F(t/\psi(t))
   ={t\eta(t)\over\psi(t)}e^{\log\psi(t)-\psi(t/\psi(t))}\to 0.
$$
\qed
\vskip4mm\noindent
{\bf Proof of Theorem \ref{thm:3.1}.} Condition \eqref{e:3.2}, $t\eta(t)$ bounded away from 0 
and $t^{2+\delta}\F(t)\to 0$ imply
$$
\limsup_{n\to\infty}\sup_{s\ge t_n}{\log n\over\psi(s)}\le{2\over 2+\delta}
$$
so that \eqref{e:2.4} follows from \eqref{e:3.1}. In addition, $\psi(t)$ must be slowly varying
so 
\begin{align*}
\limsup_{s\ge t_n}ns\eta(\lambda s)\left(\F(s/a)+F(-s/a)\right)
   &\le{na\eta(\lambda s)\over s}{\psi^{1+\delta}(s)\over s^\delta}
   (s/a)^{2+\delta}\left(\F(s/a)+F(-s/a)\right)
    \\
   &\to 0.
\end{align*}

Furthermore,
$$
\sup_{s\ge t_n}{na\mu_2(s)/s\over s\wedge\eta^{-1}(\lambda s)}\to 0.
$$
Therefore, \eqref{e:2.3} and \eqref{e:2.5} hold.
\qed
\vskip4mm\noindent
{\bf Proof of Theorem \ref{thm:3.2}.} By the dominated variation assumption, (2.1) 
holds with $b$ measurable and bounded and $\eta$ measurable and satisfying
\begin{align} \label{e:4.23}
0\le t\eta(t)\le A \text{ for all } t>0, 
\end{align}
for some finite $A$ (cf.\ Bingham, Goldie and Teugels (1987), Theorem 2.2.7).

Conditions \eqref{e:2.3} and \eqref{e:2.4} are easily checked from \eqref{e:3.4} and the assumption
${S_n\over t_n}\Povr 0$ (cf.\ Lo\`eve (1955), p. 317). It is actually not
necessary to check \eqref{e:2.5} for
this proof though in fact something stronger holds (see below). The basic 
structure in the argument for Theorem 2.1 may be followed, with a slight
difference. We focus on the relevant distinction here.

Since $\eta$ is not necessarily monotone, the bound in \eqref{e:4.19} must be modified.
In fact, using \eqref{e:4.23},
\begin{align*}
   n\int_{s/a}^m (e^{(1+\epsilon)ax\over s}-1)F(dx)
   &\le ne^{(1+\epsilon)am\over s}\F(s/a)\\
   &\le e^{-(\epsilon'-\epsilon)a+\psi(s)-\psi(s/a)}\\
   &\le e^{-(\epsilon'-\epsilon)a+B-1+A\log a}.
\end{align*}
So if we let
$$
\epsilon=\max\left(z_n(s)^{-1},{2n(|\mu_1(w)|+e^2a\mu_2(w)/s)\over
	   \delta(s\wedge\eta^{-1}(\lambda s))}\right)
$$
and
$$
\epsilon'=\epsilon+{B-1+A\log a-\log\epsilon\over a}
$$
and we use the idea in \eqref{e:3.3}, then the rest of the proof follows readily.
\qed
\vskip4mm\noindent
{\bf Proof of Theorem \ref{thm:3.3}.} We need only to check the assumptions of
Theorem \ref{thm:3.2} since $P[|X|>x]$ is regularly varying. Note that \eqref{e:2.2} 
holds. Consider the following three cases separately.

i) If $\alpha<1$, we make use of the well known
relationships (cf.\ Bingham, Goldie and Teugels (1987), Theorem
8.1.2) between $\F$, $\mu_1$ and $\mu_2$ to obtain
\begin{align} \label{e:4.24}
\limsup_{t\to\infty}{|\mu_1(t)|\over t\F(t)}\le\dfrac{(1+p)\alpha}{1-\alpha}
<\infty
\end{align}
and
$$\limsup_{t\to\infty}{\mu_2(t)\over t^2\F(t)}\le\dfrac{(1+p)\alpha}{2-\alpha}
<\infty.$$
Hence $n\F(t_n)\to\infty$ implies $S_n/t_n\Povr 0$ and 
$$
   \lim_{n\to\infty}\sup_{s\ge t_n}{na\mu_2(s)\over s^2}
   \le\dfrac{(1+p)\alpha}{2-\alpha}
      \lim_{n\to\infty}\sup_{s\ge t_n}n\F(s)(-\log n\F(s))=0.
$$
\par ii) In case $1\le\alpha<2$, the only difference is that \eqref{e:4.24} fails so
the appropriate centering of $S_n$ must be assumed. Otherwise, 
the proof is as above.
\par iii) For $\alpha\ge 2$, the extra conditions imply $S_n/t_n\Povr 0$.
Also,
$$\lim_{t\to\infty}{-\log\F(t)\over\log t}=\alpha,$$
so that for some finite $C$,
$$
-\log n\F(t) \le C\log t.
$$
Furthermore, $\mu_2(x)\log x/x^2\in RV_{-2}$ and therefore is almost decreasing
(cf.\ Bingham, Goldie and Teugels (1987), p. 41). That is,
$$
\lim_{n\to\infty}\sup_{s\ge t_n}
   {\mu_2(s)\log s/s^2\over\mu_2(t_n)\log t_n/t_n^2}<\infty.
$$
Hence \eqref{e:3.4} follows readily.
\qed
\vskip5mm\noindent
{\bf Acknowledgement.} This paper grew out of a conversation with Professor
Jozef Teugels. We are very grateful to him for his inspiration.
\vskip5mm\noindent
{\bf References.}
\baselineskip=15pt
\def\refmark{\vskip 12pt\noindent\hangindent=20pt\hangafter=1}
\refmark
Bingham, N.H., Goldie, C.M. and Teugels, J.L. (1987).
{\it Regular Variation}, Cambridge University Press.
\refmark
Chistyakov, V.P. (1964).
A theorem on the sums of independent positive random variables and its
applications to branching processes,
{\it Theor. Prob. Appl. 10}, 640--648.
\refmark
Chover, J., Ney, P. and Wainger, S. (1973a). Functions of probability measures,
{\it J. Analyse Math. 26}, 255--302.
\refmark
Chover, J., Ney, P. and Wainger, S. (1973b). Degeneracy properties of
subcritical branching processes, {\it Ann. Prob. 1}, 663--673.
\refmark
Cline, D.B.H. (1986).
Convolution tails, product tails and domains of attraction,
{\it Prob. Theor. Rel. Fields 72}, 529--557.
\refmark
Cline, D.B.H. (1987). 
Convolutions of distributions with exponential and subexponential tails,
{\it J. Austral. Math. Soc. (A) 43}, 347--365.
\refmark 
Cline, D.B.H. (1994). Intermediate regular and $\Pi$-variation,
{\it Proc. London Math. Soc. (3) 68}, 594--616.
\refmark
Embrechts, P. and Goldie, C. M. (1980). On closure and factorization properties
of subexponential distributions, {\it J. Austral. Math. Soc. (A) 29}, 243--256.
\refmark
Embrechts, P. and Goldie, C. M. (1982). On convolution tails, {\it Stoch. Proc.
Appl. 13}, 263--278.
\refmark
Goldie, C.M. (1978). Subexponential distributions and dominated variation
tails, {\it J. Appl. Prob. 15}, 440--442.
\refmark
Goldie, C.M. and Resnick, S.I. (1988).
Distributions that are both subexponential and in the domain of attraction
of an extreme-value distribution, 
{\it Adv. Appl. Prob. 20}, 706--718.
\refmark
Heyde, C.C. (1967a). 
A contribution to the theory of large deviations for sums of independent
random variables,
{\it Z. Wahr. verw. Geb. 7}, 303--308.
\refmark
Heyde, C.C. (1967b).
On large deviation problems for sums of random variables which are not
attracted to the normal law, {\it Ann. Math. Stat. 38}, 1575--1578.
\refmark
Heyde, C.C. (1968).
On large deviation probabilities in the case of attraction to a 
non-normal stable law, {\it Sanky\=a Ser. A 30}, 253--258.
\refmark
Linnik, Yu. V. (1961). 
On the probability of large deviation for the sums of independent variables,
{\it Proc. 4th Berkeley Symp. Math. Stat. Prob. 2}, 289--306.
\refmark
Lo\`eve, M. (1955). {\it Probability Theory,} Van Nostrand, New York.
\refmark
Nagaev, A.V. (1969a).
Integral limit theorems for large deviations when Cram\'er's condition in not
fulfilled I,II, {\it Theor. Prob. Appl. 14}, 51--64, 193--208.
\refmark
Nagaev, A.V. (1969b). Limit theorems for large deviations when Cram\'er's
conditions are violated (in Russian), {\it Izv. Akad. Nauk UzSSR Ser. Fiz.-Mat.
Nauk. 6}, 17--22.
\refmark
Nagaev, S.V. (1973). 
Large deviations of sums of independent random variables,
Trans. Sixth Prague Conf. Information Theory,
Random Processes and Statistical Decision Functions, Acad. Prague, 657--674.
\refmark
Nagaev, S.V. (1979).
Large deviations of sums of independent random variables,
{\it Ann. Prob. 7}, 745--789.
\refmark
Willekens, E. (1986).
Hogere Orde Theorie voor subexponenti\"ele Verdelingen, Thesis (in Dutch),
Math. Dept., Katholieke Univ., Leuven, Belgium.
\vskip2cm

\small
\begin{minipage}[t]{10cm}
Daren B. H. Cline \\
Department of Statistics \\
Texas A\&M University \\
College Station, TX 77843-3143 \\
\\
\verb|dcline@stat.tamu.edu|
\end{minipage} %
\begin{minipage}[t]{10cm}
Tailen Hsing \\
Department of Statistics \\
University of Michigan \\
Ann Arbor, MI 48109-1107 \\
\\
\verb|thsing@umich.edu|
\end{minipage}

\end{document}